\documentclass[a4paper]{amsart}

\usepackage{amsthm}
\usepackage{amssymb}
\usepackage{dsfont}
\usepackage{enumerate}
\usepackage{stmaryrd}
\usepackage{colonequals}
\usepackage[mathcal]{euscript}
\usepackage{tikz-cd} 
\usepackage[all]{xy}
\SelectTips{cm}{}

\usepackage{hyperref}
\hypersetup{%
  bookmarksnumbered=true,%
  colorlinks=true,%
  linkcolor=blue,%
  citecolor=blue,%
  filecolor=blue,%
  menucolor=blue,%
  urlcolor=blue,%
  bookmarksopen=true,%
  bookmarksdepth=2,%
  pageanchor=true}

\makeatletter
\@namedef{subjclassname@2020}{\textup{2020} Mathematics Subject Classification}
\makeatother

\numberwithin{equation}{section}

\swapnumbers

\theoremstyle{plain}
\newtheorem{theorem}[equation]{Theorem}
\newtheorem{proposition}[equation]{Proposition}
\newtheorem{lemma}[equation]{Lemma} 
\newtheorem{corollary}[equation]{Corollary}

\theoremstyle{definition}
\newtheorem{definition}[equation]{Definition}
\newtheorem{example}[equation]{Example}

\theoremstyle{remark}
\newtheorem{remark}[equation]{Remark} 

\newtheorem*{ack}{Acknowledgements}

\newcommand*{\intref}[2]{\def\tmp{#1}\ifx\tmp\empty\hyperref[#2]{\ref*{#2}}\else\hyperref[#2]{#1~\ref*{#2}}\fi}

\newcommand{\CAlg}{\mathsf{CAlg}}

\newcommand{\CSAlg}{\mathsf{CSAlg}}

\newcommand{\cosupp}{\operatorname{cosupp}}
\renewcommand{\dim}{\operatorname{dim}}
\newcommand{\End}{\operatorname{End}}
\newcommand{\Ext}{\operatorname{Ext}}
\newcommand{\tExt}{\operatorname{\rlap{$\smash{\widehat{\mathrm{Ext}}}$}\phantom{\mathrm{Ext}}}}
\newcommand{\ev}{{\mathsf{ev}}}

\newcommand{\GL}{\operatorname{GL}}

\newcommand{\Grp}{{\mathsf{Grp}}}

\newcommand{\Hom}{\operatorname{Hom}}

\newcommand{\Loc}{\operatorname{Loc}}

\newcommand{\Mod}{\operatorname{Mod}}

\newcommand{\OSp}{\operatorname{OSp}}

\newcommand{\Proj}{\operatorname{Proj}}

\newcommand{\res}{\operatorname{res}}

\newcommand{\uHom}{\underline{\Hom}}
\newcommand{\Set}{\mathsf{Set}}
\newcommand{\sgs}{{\mathsf{sgs}}}
\newcommand{\Spec}{\operatorname{Spec}}

\newcommand{\StMod}{\operatorname{StMod}}
\newcommand{\stmod}{\operatorname{stmod}}

\newcommand{\supp}{\operatorname{supp}}
\newcommand{\SVec}{\mathsf{SVec}}
\newcommand{\pisupp}{\pi\text{-}\supp}
\newcommand{\picosupp}{\pi\text{-}\cosupp}

\newcommand{\Ann}{\operatorname{Ann}}

\newcommand{\lra}{\longrightarrow}

\newcommand{\mcV}{\mathcal{V}}

\newcommand{\sfS}{\mathsf S} 
\newcommand{\sfT}{\mathsf T}

\newcommand{\bfV}{\mathbf{V}}

\newcommand{\bbA}{\mathbb A}
 
\newcommand{\bbG}{\mathbb G}
\newcommand{\bbM}{\mathbb M}
\newcommand{\bbP}{\mathbb P}
 
\newcommand{\bbZ}{\mathbb Z}

\newcommand{\fm}{\mathfrak{m}} 
\newcommand{\fp}{\mathfrak{p}}

\newcommand{\gam}{\varGamma} 
\newcommand{\lam}{\varLambda}

\newcommand{\one}{\mathds 1}

\renewcommand{\le}{\leqslant}
\renewcommand{\ge}{\geqslant}

\title[Stratification and duality]{Stratification and duality for unipotent finite supergroup schemes}

\author[Benson, Iyengar, Krause, and Pevtsova]{Dave Benson, Srikanth
  B. Iyengar, Henning Krause \\ and Julia Pevtsova}

\address{Dave Benson \\ 
Institute of Mathematics\\ 
University of Aberdeen\\ 
King's College\\ 
Aberdeen AB24 3UE\\ 
Scotland U.K.}

\address{Srikanth B. Iyengar\\ 
Department of Mathematics\\
University of Utah\\ 
Salt Lake City, UT 84112\\ 
U.S.A.}

\address{Henning Krause\\ 
Fakult\"at f\"ur Mathematik\\ 
Universit\"at Bielefeld\\ 
33501 Bielefeld\\ 
Germany.}

\address{Julia Pevtsova\\ 
Department of Mathematics\\ 
University of Washington\\ 
Seattle, WA 98195\\ 
U.S.A.}

\begin{document}
\dedicatory{To John Greenlees on his 60th birthday.}

\begin{abstract} 
We survey some methods developed in a series of papers, for classifying localising subcategories of tensor triangulated categories.
We illustrate these methods by proving a new theorem, providing such a classification in the case of the stable module category of a unipotent finite
supergroup scheme. 
\end{abstract}


\date{20 October 2020}

\maketitle

\section{Introduction}

John Greenlees' influence on mathematics is reflected throughout this volume, and our own work has benefitted enormously from his
insights. We should like to mention in particular his collaboration with May~\cite{Greenlees/May:1992a} and Dwyer~\cite{Dwyer/Greenlees:2002a}
on derived completions and local cohomology; on the local cohomology spectral sequence in the context of group cohomology~\cite{Greenlees:1995a} and his subsequent work with Lyubeznik~\cite{Greenlees/Lyubeznik:2000a}; and on duality in algebra and topology; in particular, his work with
Benson~\cite{Benson/Greenlees:1997b, Benson/Greenlees:1997a,  Benson/Greenlees:2008a}, with Dwyer and Iyengar~\cite{Dwyer/Greenlees/Iyengar:2006a}.

A broad framework for understanding localisation and duality  was developed in a series of papers~\cite{Benson/Iyengar/Krause:2009a}--\cite{Benson/Iyengar/Krause/Pevtsova:bikp7}. Originally geared towards applications to representation theory of finite groups and finite group schemes,  the framework has been applied in a number of other areas, such as commutative algebra~\cite{Dell'Ambrogio/Stevenson:2013a},  ring spectra~\cite{Barthel/Heard/Valenzuela:2018a}, modules over algebras of cochains~\cite{Benson/Greenlees:2014a, Benson/Iyengar/Krause:2011c}, and equivariant KK-theory~\cite{Dell'Ambrogio:2010a}. This theory is closely related to that of  Balmer~\cite{Balmer:2005a, Balmer:2010a, Balmer/Favi:2007a}, but the point of view is different. 

The purpose of this paper is to give an outline of the theory, by way of explaining how it applies in the case of the stable module category of a unipotent finite supergroup scheme. In particular,  we shall prove the following theorem. 

\begin{theorem}
\label{th:main}
Let $G$ be a unipotent finite supergroup scheme over a field $k$ of positive characteristic $p\ge 3$. The theory of support gives a one to one correspondence between the localising subcategories of $\StMod(kG)$ and subsets of $\Proj H^{*,*}(G,k)$.
\end{theorem}

Here $\StMod(kG)$ denotes the stable module category of $kG$-modules, $H^{*,*}(G,k)$ the cohomology ring of $G$, and  $\Proj H^{*,*}(G,k)$ its projective spectrum. Supergroup schemes are introduced Section~\ref{se:schemes} and the result above is proved in Section~\ref{se:localising}. The statement above speaks of ``the" theory of support, but in fact there are at least two, quite distinct, notions of support in this context. One is obtained from specialising the support theory introduced in \cite{Benson/Iyengar/Krause:2009a}, which is based on cohomology and applies to any compactly generated triangulated category with a ring action; in this case $\StMod(kG)$ with the canonical action of $H^{*,*}(G,k)$. The other notion of support is inspired by the theory of $\pi$-points introduced by Friedlander and Pevtsova~\cite{Friedlander/Pevtsova:2007a} in the context of finite group schemes. That these two notions of support coincide for modules over finite supergroup schemes is one of the crucial steps in proving the theorem above. 

Theorem~\ref{th:main} was first proved,  with no restrictions on the prime $p$, for finite groups~\cite{Benson/Iyengar/Krause:2011a}, using ideas from commutative algebra, mainly the Bernstein-Gelfand-Gelfand correspondence.  A new proof, from a more representation theoretic perspective, was given in \cite{Benson/Iyengar/Krause/Pevtsova:2017a}, exploiting ideas from \cite{Friedlander/Pevtsova:2007a} and a concept of cosupport for representations, which is related to derived completions introduced by Greenlees and May~\cite{Greenlees/May:1992a}. These ideas also played a pivotal role in our proof of Theorem~\ref{th:main} for the case of finite group schemes~\cite{Benson/Iyengar/Krause/Pevtsova:2018a}.

Theorem~\ref{th:main} is a culmination of our work reported in \cite{Benson/Iyengar/Krause/Pevtsova:bikp5, Benson/Iyengar/Krause/Pevtsova:bikp7, Benson/Pevtsova:2020a}. We sketch the details of how to put these pieces together in Section~\ref{se:homeo}.  This work take as a starting point the applications of the theory of super polynomial functors to representations and cohomology of finite  supergroup schemes, developed in a series of papers by Drupieski and Kujawa~\cite{Drupieski:2016a, Drupieski/Kujawa:2019a,Drupieski/Kujawa:2019c}.  Their work~\cite{Drupieski/Kujawa:1811.04840} represents a  development parallel to our own.  The main difference is that we treat both  connected and non-connected cases of unipotent supergroup schemes  and are interested in infinite dimensional $kG$-modules. In  \cite{Drupieski/Kujawa:1811.04840} one can find many intricate results,  examples and calculations for finite dimensional modules over connected group schemes.

Finally, in Section~\ref{se:duality} we give a brief overview of the local duality for $\StMod(kG)$, summarised in the following result.

\begin{theorem}
\label{th:duality}
Let $G$ be a finite supergroup scheme over a field $k$ of positive characteristic $p\ge 3$. For each $\fp\in \Proj H^{*,*}(G,k)$, the corresponding  localising subcategory $\gam_\fp\StMod(kG)$ satisfies local duality.
\end{theorem}

Observe that, in contrast with Theorem~\ref{th:main}, here $G$ is not restricted to be unipotent. For finite groups, such a statement was conjectured by Benson~\cite{Benson:2001a}  and proved by Benson and Greenlees~\cite{Benson/Greenlees:2008a}; see also \cite{Benson:2008a}. In \cite{Benson/Iyengar/Krause/Pevtsova:2019a} we gave a new proof of their result covering also the case of finite group schemes. It turns out that local duality is a general feature of Gorenstein algebras, and can be viewed as an extension of a duality theorem due to Auslander~\cite{Auslander:1978a} and Buchweitz~\cite{Buchweitz:1986a}  generalising Tate duality for finite groups. All this is explained in \cite{Benson/Iyengar/Krause/Pevtsova:bikp6}. One can view the result above as a variation, dealing with graded Gorenstein algebras.
 
In what follows, we have tried to present the development of the theory from elementary abelian $p$-groups to finite groups to finite group schemes to finite supergroups schemes. We focus mostly on stratification, and even then cannot hope to present a complete account of the theory in a short survey.
Thus the aim is to give just a flavour of numerous techniques and developments that happened in the theory of support varieties in (relatively) recent
years many of which were influenced by John Greenlees' work.

\begin{ack}
  It is a pleasure to acknowledge the support provided by the American
  Institute of Mathematics in San Jose, California, through their
  ``Research in Squares" program.  We also acknowledge the National Science
  Foundation under Grant No.\ DMS-1440140 which supported three of the
  authors (DB, SBI, JP) while they were in residence at the Mathematical Sciences
  Research Institute in Berkeley, California, during the Spring 2018
  semester.  Finally, two of the authors (DB, JP) are grateful for hospitality
  provided by City, University of London.

SBI was partly supported by NSF grant DMS-2001368. JP was partly supported by NSF grants DMS-1501146,  DMS-1901854, and a Brian and Tiffinie Pang faculty fellowship.
\end{ack}

\section{Local cohomology and support}
\label{se:local} 

Let $\sfT$ be a  triangulated category, with shift $\Sigma$, admitting all small coproducts, and compactly generated. The examples we have in mind come from algebra, topology and geometry. The \emph{centre} $Z^*(\sfT)$ of $\sfT$ is the graded ring whose degree $n$ part consists of the natural
transformations $\eta$ from the identity functor to $\Sigma^n$ satisfying $\eta\Sigma=(-1)^n\Sigma\eta$. These form a graded commutative ring which is usually not Noetherian. We assume that we are given a graded commutative Noetherian ring $R$ together with a homomorphism of rings $R\to
Z^*(\sfT)$. This is called an \emph{action} of $R$ on $\sfT$. For $X,Y$ in $\sfT$ set
\[
\Hom^{*}_{\sfT}(X,Y) \colonequals \bigoplus_{n\in\bbZ} \Hom_{\sfT}(X,\Sigma^n Y)\,;
\]
this has a structure of an $R$-module, compatible with morphisms in $\sfT$.

A subset $V$ of a topological space $X$ is \emph{specialization  closed} if $V$ contains the closure of its points.
We write $\Spec R$ for   the homogeneous primes ideals of $R$, with the Zariski topology.  For each specialisation closed subset $V$ of $\Spec R$, the subcategory $\sfT_V$ of \emph{$V$-torsion} objects in $\sfT$ is the full subcategory consisting  of the objects $X$ such that  $\Hom^*_\sfT(C,X)_\fp=0$
for each compact object $C$ in $\sfT$  and each $\fp\not\in V$. For $X$ in $\sfT$ there is a functorial triangle
\[ 
\gam_VX \to X \to L_VX
\]
where $\gam_VX$ is in $\sfT_V$, and $L_VX$ admits no non-zero maps from any object in $\sfT_V$.

\begin{definition}
\label{def:Gamma-p}
Given a $\fp$ in $\Spec R$ choose specialisation closed subsets $V$ and $W$ with $V \not\subseteq W$  and $V \subseteq W\cup \{\fp\}$, and define 
\[ 
\gam_\fp X \colonequals \gam_V L_WX = L_W\gam_VX. \]
\end{definition}

This turns out to be independent of choice of $V$ and $W$ satisfying these conditions, and defines an idempotent functor $\gam_\fp\colon \sfT\to\sfT$ which may be thought of as isolating the layer of $\sfT$ corresponding to the prime $\fp$, and consisting of the objects ``supported at $\fp$.''  The functor $\gam_\fp$ has a right adjoint, which we denote $\lam^{\fp}$; it  has to do with completions along $\fp$, and plays an equally important role in the theory;  for details, see~\cite{Benson/Iyengar/Krause:2009a, Benson/Iyengar/Krause:2012a}. 

\begin{definition}
\label{def:biksupp}
The \emph{support} and \emph{cosupport} of an object $X$ in $\sfT$ are the subsets
\begin{align*}
\supp_R(X) &\colonequals \{\fp\in \Spec R\mid \gam_\fp X\ne 0\}\\
\cosupp_R(X) &\colonequals \{\fp\in \Spec R\mid \lam^\fp X\ne 0\}
\end{align*}
It is not hard to prove that both support and cosupport detect zero objects:  an object $X$ in $\sfT$ is $0$ if and only if $\supp_R(X)=\varnothing$, if and only if $\cosupp_R(X)=\varnothing$.
\end{definition}

We focus mostly on support, but cosupport resurfaces in Section~\ref{se:localising}.  We write $\supp_R(\sfT)$ for $\bigcup \supp_R(X)$, where $X$ ranges over the objects in $\sfT$.

Under mild conditions, such as $R$ having finite Krull dimension,  the localising subcategory $\Loc_\sfT(X)$ of $\sfT$ generated by an object $X$ is equal to the localising subcategory $\Loc_\sfT(\{\gam_\fp X \mid \fp \in \supp_R(\sfT)\})$ generated by the objects $\gam_\fp X$. If this is the case, we say that the \emph{local-global principle} holds. Then there is a one to one correspondence between localising subcategories of $\sfT$ and functions assigning to
each $\fp$ in $\supp_R(\sfT)$ a localising subcatgory of $\gam_\fp\sfT$. The function corresponding to a localising subcategory $\sfS$ sends $\fp$ to $\sfS\cap\gam_\fp\sfT$. This is described in detail in~\cite{Benson/Iyengar/Krause:2011a}.

We say that an action of a graded commutative ring $R$ on $\sfT$ \emph{stratifies} $\sfT$ if the local-global principle holds and for each $\fp$ in $\supp_R(\sfT)$, the subcategory $\gam_\fp\sfT$ is a minimal with respect to inclusion, among localising subcategories of $\sfT$. Under these circumstances, there is one to one correspondence between localising subcategories of $\sfT$  and the subsets of $\supp_R(\sfT)$. The subset corresponding to a localising subcategory $\sfS$ is its support, $\supp_R\sfS$, by which we mean the set of primes $\fp$ for which  $\gam_\fp$ is not the zero functor on $\sfS$.

\begin{example}
\label{eg:StModkG}
Let $G$ be a finite group and $k$ a field of characteristic $p$ dividing the order of $G$. The stable module category $\StMod(kG)$ is a compactly generated triangulated category, with compact objects the finite dimensional $kG$-modules, and there is an action of the cohomology ring $H^*(G,k)$. In this case, we write $\supp_G(M)$ for $\supp_{H^*(G,k)}(M)$.

The support of $\StMod(kG)$ is $\Proj H^*(G,k)$, which is the set $\Spec H^*(G,k)$ excluding the  maximal ideal consisting of all elements of positive degree. When $M$ is finite dimensional,  $\supp_G(M)$ is the closed subset of $\Proj H^*(G,k)$ defined by $\Ann_{H^*(G,k)}\Ext^{*,*}(M,M)$.

It is proved  in~\cite{Benson/Iyengar/Krause:2011b} that when $G$ is a $p$-group this action stratifies $\StMod(kG)$.
\end{example}

Such a stratification does not hold for general finite groups, because the localising subcategories $\gam_\fp\StMod(kG)$ are \emph{tensor ideal}, namely they are closed under tensor products with all modules, whereas there may be localising subcategories that are not. For a  finite $p$-group, the only simple module is the trivial module $k$, so the localising subcategory it generates is the whole of $\StMod(kG)$, and hence all localising subcategories are tensor ideal. To address arbitrary finite groups, we must take the tensor structure into account. 

\subsection*{Tensor triangulated categories}
Suppose now that $\sfT$ is a compactly generated triangulated category that comes with a symmetric monoidal tensor product $\otimes \colon \sfT \times \sfT \to \sfT$, is exact in each variable, preserving small coproducts,  and with unit $\one$. In this situation, we say that $\sfT$ is a \emph{tensor triangulated category}. There is a particularly good kind of action of a graded commutative ring $R$ on such a $\sfT$;  namely, one that factor as $R \to \End^*_\sfT(\one)$ followed by the natural map $\End^*_\sfT(\one)\to Z^*(\sfT)$ induced by the tensor product. We say that such an action is \emph{canonical}.

Given a canonical action of $R$ on $\sfT$, the functor $\gam_\fp$ is naturally isomorphic to $\gam_\fp\one \otimes -$. The subcategory $\gam_\fp\sfT$ is \emph{tensor ideal}, meaning that it is closed under tensoring with arbitrary objects in $\sfT$. Denote by $\Loc^\otimes_\sfT(X)$ the tensor ideal localising
subcategory of $\sfT$ generated by an object $X$. The tensor version of the local-global principle holds in this situation, and says that $\Loc^\otimes_\sfT(X)$ is equal to the tensor ideal localising subcategory $\Loc^\otimes_\sfT(\{\gam_\fp X \mid \fp \in \supp_R(\sfT)\})$. So we get a one to one correspondence between tensor ideal localising subcategories of $\sfT$ and functions assigning to each $\fp$ in $\supp_R(\sfT)$ a tensor ideal localising subcategory of
$\gam_\fp\sfT$. 

We say that a  canonical action of a graded commutative ring $R$  \emph{stratifies} a tensor triangulated category $\sfT$ if for each $\fp\in\supp_R(\sfT)$, the subcategory  $\gam_\fp\sfT$ is minimal as a tensor ideal localising subcategory. See also the work of Hovey, Palmieri, and Strickland~\cite{Hovey/Palmieri/Strickland:1997a} where such a minimality plays a prominent role.

\begin{example}
Let $G$ be a finite group. The stable module category of $kG$-modules is a tensor triangulated category, with tensor unit the trivial module $k$. The
Tate cohomology ring $\widehat H^*(G,k)$ is none other than the graded ring  $\End^*_{\StMod(kG)}(k)$,  and so  has a canonical action on $\StMod(kG)$. 
While $\widehat H^*(G,k)$ is usually not Noetherian, its subring  $H^*(G,k)$ is, which then inherits the canonical action on $\StMod(kG)$. 
\end{example}

The following theorem is proved in~\cite{Benson/Iyengar/Krause:2011b}. 

\begin{theorem}
\label{th:stratification}
Let $G$ be a finite group and $k$ a field of characteristic $p$. As a tensor triangulated category, the action of $H^*(G,k)$ stratifies $\StMod(kG)$.
In particular, there is a bijection, defined by $\supp_G(-)$, between tensor ideal localising subcategories of $\StMod(kG)$ and subsets of $\Proj H^*(G,k)$.\qed
\end{theorem}

A key step in the proof is a reduction to elementary abelian groups, which depends on the Quillen stratification theorem~\cite{Quillen:1971a,Quillen:1971b}, and work of Chouinard, recalled below.


\section{The rank variety} 
\label{se:rank}
For a finite group $G$, there is another description of support that is much better suited to computation.  Chouinard~\cite{Chouinard:1976a} proved that over an arbitrary commutative ring of coefficients $k$, a $kG$-module is projective if and only if its restriction to every elementary abelian subgroup $E$ of $G$ is a projective $kE$-module. If $k$ is a field of characteristic $p$, we only need to consider elementary abelian $p$-subgroups. In this case, when  $k$ is algebraically closed, Dade~\cite{Dade:1978b} proved that a  finite dimensional $kE$-module is projective if and only if its restriction to each cyclic shifted subgroup is projective. The precise statement is as follows. If $E\colonequals \langle g_1,\dots,g_r \rangle \cong (\bbZ/p)^r$, is an elementary abelian $p$-group,  set $X_i=g_i-1$ in $kE$  for $1\le i \le r$. Then $kE$ is a truncated polynomial ring:
\[
kE\cong \frac{k[X_1,\dots,X_r]}{(X_1^p,\dots,X_r^p)}\,.
\]
Let $J(kE)$ be the radical, $(X_1,\dots,X_r)$, of $kE$. For  $\lambda = (\lambda_1,\dots,\lambda_r)$ in $\bbA^r(k)$ set
\[ 
X_\lambda\colonequals \lambda_1 X_1 + \dots + \lambda_r X_r \quad\text{in $kE$} 
\]
and $\alpha_\lambda\colon k[t]/(t^p) \to kE$ for the homomorphism of $k$-algebras sending $t$ to $X_\lambda$. Given a $kE$-module $M$, we write $\alpha_\lambda^*(M)$ for the $k[t]/(t^p)$-module obtained by restriction along $\alpha_\lambda$.

\begin{theorem}[Dade~\cite{Dade:1978b}]
Let $k$ be an algebraically closed field of characteristic $p$, and $E$ an elementary abelian $p$-group. A finite dimensional dimensional $kE$-module $M$ is projective if and only if  $\alpha_\lambda^*(M)$ is projective for all $0\ne\lambda\in\bbA^r(k)$.\qed
\end{theorem}

Based on this, Carlson~\cite{Carlson:1983a} introduced the notion of rank variety. 

\begin{definition}
Let $k$ be an algebraically closed field of characteristic $p$, and $E$ an elementary abelian $p$-group. The \emph{rank variety} of a finitely generated $kE$-module $M$ is  the subset
\[ 
V^r_E(M)\colonequals  \{\lambda\in\bbA^r(k)\mid \alpha_\lambda^*(M) \textrm{ is not  projective}\}
\]
of $\bbA^r(k)$. We write $V^r_{kE}(M)$ if the field of coefficients needs to be specified. Observe that $V^r_E(M)$ contains $0$ and is homogenous. It is also closed, for  projectivity over $k[t]/(t^p)$ is detected by a rank condition on the operator representing the action of $t$. This also  gives a method for calculating  equations defining $V^r_E(M)$. Ostensibly, this depends on the choice of generators for $E$ as an elementary abelian group. However, if $\alpha,\beta\colon k[t]/(t^p) \to kE$ are maps such that $\alpha(t)$ and $\beta(t)$ have the same image in $J(kE)/J^2(kE)$ then $\alpha^*(M)$ is projective if and only if $\beta^*(M)$ is projective. So it makes sense to think of the ambient affine space  $\bbA^r(k)$ as identified with $J(kE)/J^2(kE)$. We return to this  ambiguity in Section~\ref{se:schemes} when we discuss $\pi$-points. 
\end{definition}

The cohomology ring of $E$ is well-known and easy to compute:
\[
H^*(E,k) = 
\begin{cases}
k[y_1,\dots,y_r]  & \text{if $p=2$}\\
k[x_1,\dots,x_r] \otimes \Lambda(y_1,\dots,y_r)  & \text{for $p$ odd}
\end{cases}
\]
where $y_i$ is in degree one the $x_i$ is in degree two, and is the Bockstein $\beta(y_i)$ of $y_i$.

Let $k[Y_1,\dots,Y_r]$ be the coordinate ring of $\bbA^r(k)$.  If $p=2$, this can be identified with $H^*(E,k)$ in such a way that the $Y_i$  correspond to the $y_i$. If $p$ is odd, there is a twist:  we  have to identify $x_i$ with $Y_i^p$, so that there is a Frobenius twist involved in the identification of  $\Spec H^*(E,k)$ with $\Spec k[Y_1,\dots,Y_r]$. Carlson conjectured that  this identifies support and rank variety for a finite dimensional  $kE$-module.

\begin{theorem}[Avrunin, Scott~\cite{Avrunin/Scott:1982a}] 
\label{th:AS} 
Let $k$ be an algebraically closed field. Under the usual identification of radical homogeneous ideals in  $\Spec H^*(E,k)$  with affine cones in $\bbA^r(k)$, the radical ideal defining the subset $\supp_E(M)$ corresponds to $V^r_E(M)$ for  any finite-dimensional $kE$-module $M$. 
\end{theorem}

Given an extension of fields $k\subseteq K$ and a $kE$-module $M$, let $M_K$ be the $KE$-module $K\otimes_k M$. Benson, Carlson and Rickard~\cite{Benson/Carlson/Rickard:1996a}  extended Dade's theorem to cover  infinitely generated modules, as follows:

\begin{theorem}
\label{th:BCR}
A $kE$-module $M$ is projective if and only if $\alpha^*_{\lambda}(M_K)$ is projective for all extension fields $K$  of $k$ and all $0\ne \lambda$ in $\bbA^r(K)$.
\end{theorem}

For  $0\ne\lambda\in\bbA^r(K)$ let $\fp\subseteq k[Y_1,\dots,Y_r]$ be the ideal consisting of  homogeneous polynomials that vanish at $\lambda$. This is a prime ideal and $\lambda$ is a \emph{generic} point for $\fp$. Each $\fp$ in $\Proj k[Y_1,\dots,Y_r]$ occurs this way. If $\lambda\in\bbA^r(K)$ and  $\lambda'\in\bbA^r(K')$ are generic points for the same prime ideal, then $\alpha^*_{\lambda}(M_K)$ is projective if and only if $\alpha^*_{\lambda'}(M_{K'})$ is projective.  So instead of a rank variety, we assign to $M$ a subset of  $\Proj k[Y_1,\dots,Y_r]$. 

\begin{definition}
\label{def:mcV}
Set $\mcV_E^r(M)$ to be the set of $\fp\in\Proj k[Y_1,\dots,Y_r]$ such that if $\lambda\in \bbA^r(K)$ is generic for $\fp$ then $\alpha^*_{\lambda}(M_K)$ is not projective.
\end{definition}

Historically the rank variety was defined as a subset of $\bbA^r(k)$,  but we switch to considering projective varieties since all our constructions are  ``invariant" under scalar multiplication.  It follows from the main theorems of~\cite{Benson/Carlson/Rickard:1996a}  that the subset $\supp_E(M)\subseteq\Proj H^*(E,k)$  corresponds to $\mcV_E^r(M)$ for \emph{any} $kE$-module $M$, once the  appropriate identification, involving Frobenius for $p>2$,  of $\Proj H^*(E,k)$ and  $\bbP^{r-1}$ is made.  

We shall see in Section~\ref{se:pi} that the appropriate  generalisation of these concepts to finite group schemes leads to the theory of $\pi$-points.

\section{Finite group schemes}
In this section, we introduce finite group schemes, following the approach given in~\cite[Chapter 1]{Jantzen:2003a}. Throughout $k$ will be a field. 

An \emph{affine scheme} over $k$ is a representable functor from the category $\CAlg(k)$ of commutative algebras over $k$ to sets. 
The representing object of an affine scheme $S$, denoted $k[S]$, is called its \emph{coordinate ring}. Thus $S \colon \CAlg(k)\to \Set$ takes the form $\Hom_{\CAlg(k)}(k[S],-)$.

An \emph{affine group scheme} is a functor $G\colon \CAlg(k) \to \Grp$ whose composite with the forgetful functor $\Grp\to\Set$ is representable. By Yoneda's lemma, the natural transformations given by the group operations give rise to a  structure on $k[G]$ of commutative Hopf algebra. This gives a contravariant equivalence of categories from affine group schemes to commutative Hopf
algebras, sending $G$ to $k[G]$.

A \emph{finite group scheme} is an affine group scheme $G$ with the property that the coordinate ring $k[G]$ is finite dimensional over $k$. In this case, we may dualise to get the \emph{group algebra} $kG \colonequals \Hom_k(k[G],k)$, which has the structure of cocommutative Hopf algebra. This gives a covariant equivalence of categories from finite group schemes to finite dimensional cocommutative Hopf algebras, sending $G$ to $kG$.

Finite groups are examples of finite group schemes, but there are many more, including $p$-restricted Lie algebras, and Frobenius kernels of affine group schemes. See, for example, ~\cite[\S1]{Benson/Iyengar/Krause/Pevtsova:2017a} for examples and an explanation of how finite groups fit into the context. 

Friedlander and Suslin~\cite{Friedlander/Suslin:1997a} proved that for any finite group scheme $G$, the $k$-algebra $H^*(G,k)$ is Noetherian; moreover, the $H^*(G,k)$-module $H^*(G,M)$ is finitely generated for any finite dimensional $kG$-module $M$. This opened the door for the development  of support theories for finite group schemes. Another landmark development in this area was the theory of 
$\pi$-points.


\section{\texorpdfstring{The theory of $\pi$-points}{The theory of Ï-points}}\label{se:pi}

The theory of $\pi$-points for finite group schemes was initiated by Friedlander and Pevtsova~\cite{Friedlander/Pevtsova:2005a,Friedlander/Pevtsova:2007a}, and generalises rank varieties discussed in Section~\ref{se:rank}. This approach does not rely on the 
choice of generators needed to define cyclic shifted subgroups for elementary abelian $p$-groups which makes it applicable in a much 
greater generality. 

Let $G$ be a finite group scheme over a field $k$ of positive characteristic $p$. A \emph{$\pi$-point} $\alpha$ of $G$ is 
given as follows. We choose an extension field $K$ of $k$, a unipotent abelian subgroup scheme $E$ of $G_K$, and a flat map
\[ 
\alpha\colon K[t]/(t^p) \to KE \subseteq KG_K\,. 
\] 
Note that $E$ does not have to come from a subgroup scheme of $G$ by extension of scalars. The flatness condition is equivalent to the statement that the image of $t$ is in $J(KE)$ but not in $J^2(KE)$. 

Given such an $\alpha$ consider the composite
\[ 
H^*(G,k) \subseteq K\otimes_k H^*(G,k)\cong H^*(G_K,K)\cong 
\Ext^*_{KG_K}(K,K) \xrightarrow{\alpha^*} \Ext^*_{K[t]/(t^p)}(K,K)\,. 
\]
The ring $\Ext^*_{K[t]/(t^p)}(K,K)$ is isomorphic to $K[v]$ with $|v|=1$ if $p=2$, and to $K[u,v]/(u^2)$ with $|u|=1$ and $|v|=2$ if $p\ne 2$. The nil radical is zero in the first case, and the ideal $(u)$  in the second case. We define $\fp(\alpha)$ to be the inverse
image in $H^*(G,k)$ of the nil radical of $\Ext^*_{K[t]/(t^p)}(K,K)$ under the above map. This is a homogeneous prime ideal in $H^*(G,k)$.
The following theorem is due to Friedlander and Pevtsova~\cite{Friedlander/Pevtsova:2007a}.

\begin{theorem}
\label{th:pi-equiv}
If $\alpha\colon K[t]/(t^p) \to KG_K$ and $\beta\colon L[t]/(t^p)\to LG_L$ are $\pi$-points, the following conditions are equivalent.
\begin{enumerate}[\quad\rm(i)]
\item 
$\fp(\alpha)=\fp(\beta)$;
\item 
$\alpha^*(M_K)$ is projective if and only if $\beta^*(M_L)$ is projective for a $kG$-module $M$; 
\item
$\alpha^*(M_K)$ is projective if and only if $\beta^*(M_L)$ is projective for any finite dimensional $kG$-module $M$.\qed
\end{enumerate}
\end{theorem}

We therefore put an equivalence relation on the set of $\pi$-points of $G$, where $\alpha\sim\beta$ if and only if the equivalent
conditions of the theorem hold. The map sending the equivalence class of a $\pi$-point $\alpha$ to the prime $\fp(\alpha)$ gives a bijection 
between the equivalence classes of $\pi$-points in $G$ and the set  $\Proj H^*(G,k)$. 

With this equivalence relation, every $\pi$-point is equivalent to one that factors through a subgroup scheme which is not only abelian
unipotent, but elementary. Let $\bbG_{a}$ be the additive group scheme and for each integer $r\ge 0$, let $\bbG_{a(r)}$ denote its $r$th
Frobenius kernel; see~\cite[Chapter~9]{Jantzen:2003a}.

\begin{definition} 
\label{de:elementary}
A finite group scheme is \emph{elementary} if it is isomorphic to the group scheme $\bbG_{a(r)} \times (\bbZ/p)^s$ with $r,s\ge 0$.
\end{definition}

These group schemes are called ``quasi-elementary"  by Bendel~\cite{Bendel:2001a}. For a finite group scheme  cohomology is detected, modulo nilpotents, on elementary subgroup schemes over extension fields. This explains their central role  in this theory.

\begin{definition}
Let $G$ be a finite group scheme over a field $k$. The \emph{$\pi$-support}, denoted $\pisupp_G(M)$, of a $kG$-module $M$ is the subset of $\Proj H^*(G,k)$ consisting of  primes $\fp(\alpha)$, for $\alpha\colon K[t]/(t^p)\to KG_K$ a $\pi$-point such that $\alpha^*(M_K)$ is not projective.
\end{definition}

The $\pi$-support is a ``generator-invariant" generalisation of Carlson's rank variety for any finite group scheme. 
On the other hand, exactly as for finite groups, one has a canonical
action of $H^*(G,k)$ on $\StMod(kG)$ and hence a notion of support for
$kG$-modules; see Section~\ref{se:local}, especially
\ref{eg:StModkG}. The following theorem from
\cite{Friedlander/Pevtsova:2007a,Benson/Iyengar/Krause/Pevtsova:2018a}
reconciles these two notions.  In doing so, it puts the results for
elementary abelian $p$-groups~\cite{Avrunin/Scott:1982a}, finite
groups~\cite{Benson/Carlson/Rickard:1996a}, restricted Lie algebras~\cite{Friedlander/Parshall:1986a}, and infinitesimal group schemes~\cite{Bendel/Friedlander/Suslin:1997b, Pevtsova:2002a} into a uniform statement.

\begin{theorem}
 One has $\supp_G(M)=\pisupp_G(M)$ for any $kG$-module $M$.\qed
\end{theorem}

\section{Finite supergroup schemes}
\label{se:schemes}

The definition of affine superschemes is parallel to that of affine group schemes. It is obtained replacing the underlying category of vector spaces over $k$ by the category of super vector spaces. A \emph{super vector space} over $k$ is a $\bbZ/2$-graded vector space $V=V_0\oplus V_1$. The tensor product of two such is given by 
\[
(V\otimes W)_0=V_0\otimes W_0\,\oplus\, V_1\otimes W_1\quad\text{and}\quad  (V\otimes W)_1=V_0\otimes W_1\,\oplus\, V_1\otimes W_0\,.
\]
This tensor product has a symmetric braiding $V\otimes W \cong W\otimes V$ given by sending $v\otimes w$ to $(-1)^{|v||w|}w\otimes v$, where
$|v|$ denotes $0$ if $v\in V_0$ and $1$ if $v\in V_1$. Thus the category $\SVec(k)$ of super vector spaces is a symmetric monoidal
abelian category. A \emph{commutative superalgebra} over $k$ is a commutative algebra in this category. Thus it consists of an object $A$ together with a multiplication $A\times A \to A$ which is commutative with respect to the symmetric braiding,  associative, and unital. These form a category $\CSAlg(k)$.

An \emph{affine superscheme} over $k$ is a representable functor from commutative superalgebras over $k$ to sets. The representing object of
an affine superscheme $S$ is the coordinate ring $k[S]$.  An \emph{affine supergroup scheme} is a functor $G\colon \CSAlg(k) \to \Grp$ whose composite with the forgetful functor $\Grp\to\Set$ is representable. The representing object, the coordinate ring $k[G]$ of $G$, is a commutative Hopf superalgebra. This way, we obtain a contravariant equivalence of categories from affine supergroup schemes to commutative Hopf superalgebras. A Hopf superalgebra need not be a Hopf algebra, for the diagonal map need not be a map of  ungraded algebras.

A \emph{finite supergroup scheme} is an affine supergroup scheme $G$ with the property that the coordinate ring $k[G]$ is finite dimensional over $k$. In this case, we may dualise to get the \emph{group algebra} $kG = \Hom_k(k[G],k)$, which has the structure of cocommutative Hopf superalgebra. This way, we obtain a covariant equivalence of categories from finite supergroup schemes to finite dimensional cocommutative Hopf superalgebras.

Examples of finite supergroup schemes include finite groups, finite group schemes, exterior algebras, as well as Frobenius kernels of affine supergroup schemes such as the general linear ones $\GL(a|b)$ and the orthosymplectic ones $\OSp(a|2b)$. We write $\bbG_a^-$ for the finite supergroup scheme whose group algebra is an exterior algebra on one primitive element, in degree $1$. This is the simplest example of a finite supergroup scheme which is not a finite group scheme.

A finite supergroup scheme $G$ is \emph{unipotent} if the kernel of the augmentation map $kG \to k$ is equal to the nil radical. This is equivalent to $kG$ having exactly two simple modules, the trivial module in even degree and the trivial module in odd degree. For example, $\bbG_a^-$ is unipotent.

The module category for a finite supergroup scheme is an abelian category, with a parity change functor $\Pi$. The action of an element $a\in kG$ on $\Pi m \in \Pi M$ is given by $a.\Pi m = (-1)^{|a|}\Pi(am)\in \Pi M$. The stable module category $\StMod(kG)$ is a $\bbZ/2$-graded  triangulated category. So it comes with an internal shift $\Pi$ whose square is naturally isomorphic to the identity, and a cohomological shift $\Omega^{-1}$. The distinguished triangles
\[ 
M_1 \to M_2 \to M_3 \to \Omega^{-1}M_1 
\]
are those isomorphic to triangles coming from short exact sequences of $kG$-modules via a pushout diagram
\[ 
\xymatrix{0 \ar[r] & M_1 \ar[r]\ar@{=}[d] & M_2 \ar[r]\ar[d] & M_3 \ar[r]\ar[d] & 0 \\ 
0 \ar[r] & M_1 \ar[r] & I \ar[r] & \Omega^{-1}M_1 \ar[r] & 0} 
\]
where $I$ is an injective module into which $M_2$ embeds.

For a $\bbZ/2$-graded triangulated category $\sfT$, we need a slight modification to the definition of centre, to take account of the grading. We define $Z^{n,j}(\sfT)$ ($n\in\bbZ$, $j\in\bbZ/2$) to be the natural transformations $\eta$ from the identity functor to $\Sigma^n\Pi^j$ satisfying  $\eta\Sigma=(-1)^n\Sigma\eta$ and $\eta\Pi=(-1)^j\Pi\eta$. These form a $\bbZ\times \bbZ/2$-graded ring. It is graded commutative, in the sense that if $x\in Z^{m,i}(\sfT)$ and $y\in Z^{n,j}(\sfT)$ then $yx = (-1)^{mn}(-1)^{ij}xy$.   The cohomology ring $H^{*,*}(G,k)$ is also $\bbZ\times\bbZ/2$-graded
commutative in this sense. For example, $H^{*,*}(\bbG_a^-,k)$ is a polynomial ring on a single generator in
degree $(1,1)$.

There is a canonical action of $H^{*,*}(G,k)$ on $\StMod(kG)$, so it makes sense to try to stratify $\StMod(kG)$ as a $\bbZ/2$-graded  tensor triangulated category using this action.  The definition of localising subcategory needs to be modified to take account of the extra grading; we only consider localising
subcategories closed under the operation $\Pi$. 

\begin{definition}
\label{def:ProjH**,Hbullet}
We define $\Proj H^{*,*}(G,k)$ to consist of the prime ideals which are homogeneous
with respect to both gradings, with the Zariski topology,  and as usual we exclude the maximal
ideal of elements generated by homogeneous elements whose degree 
is not $(0,0)$. Since elements of degree (even,odd) or (odd,even)
square to zero, these elements are contained in every prime ideal.
Modulo these elements, the $\bbZ/2$-grading is just the mod two
reduction of the $\bbZ$-grading, and so we can ignore it. So we write
$H(G,k)$ for the singly graded ring whose degree $i$ component
is the degree $(i,0)$ or $(i,1)$ component of $H^{*,*}(G,k)$ according
as $i$ is even or odd. This is a graded ring which becomes 
strictly commutative if we reduce modulo nilpotents, and
at the level of topological spaces, $\Proj H(G,k)$ may be identified with
$\Proj H^{*,*}(G,k)$. Thus for example $H(\bbG_a^-,k)$ is a
polynomial ring on a generator in degree one.
\end{definition}

\begin{definition}
\label{def:supp} 
For a $kG$-module $M$, let $\supp_G(M)$ denote the support defined via the action of $H(G,k)$ on $\StMod(kG)$, as in Section~\ref{se:local}. 
\end{definition} 

\begin{remark} \label{rk:classical} 
By definition, $\supp_G(M)$ is a subset of $\Proj H(G,k)$. Recall that when $M$ is  finite-dimensional, $\supp_G(M)$ is a closed subset defined by  the annihilator of  $\Ext^{*,*}_{kG}(M,M)$ as a module over $H(G,k)$. Moreover, if $G$ is unipotent,  we can take the annihilator of  $H^{*,*}(G,M)$. 
\end{remark}

\section{Elementary  supergroup schemes}\label{se:elem}

The definition of elementary  supergroup schemes is more complicated than for finite group schemes, and we have so far only addressed the unipotent case. Drupieski and Kujawa~\cite{Drupieski/Kujawa:2019a, Drupieski/Kujawa:2019b, Drupieski/Kujawa:2019c}  suggest that similar definitions may suffice for arbitrary finite supergroup schemes.

In~\cite{Benson/Iyengar/Krause/Pevtsova:bikp5} we construct a family of finite connected unipotent  supergroup schemes $E_{m,n}^-$
with $m,n\ge 1$ related to the Witt vectors and declare a supergroup
scheme to be \emph{elementary} if it is isomorphic to a quotient of
some $E_{m,n}^-\times (\bbZ/p)^s$.  These are classified; see
\cite[Remark 8.14]{Benson/Iyengar/Krause/Pevtsova:bikp5} and compare with Definition~\ref{de:elementary}.
\begin{theorem}
\label{th:elemclass}
Each elementary  supergroup scheme is isomorphic to one of:
\begin{enumerate}[\quad\rm I.]
\item $\bbG_{a(n)} \times (\bbZ/p)^s$ with $n,s\ge 0$,
\item $\bbG_{a(n)} \times \bbG_a^- \times (\bbZ/p)^s$ with $n,s\ge 0$,
\item
\begin{enumerate}[\rm(i)]
\item[(i)]  $E_{m,n}^- \times (\bbZ/p)^s$ with $m\ge 2$, $n\ge 1$, $s\ge 0$,
\item[(ii)] $E_{m,n,\mu}^-\times (\bbZ/p)^s$ with $m,n\ge 1$, $0\ne \mu\in
  k^\times/(k^\times)^2$  and $s\ge 0$.
  \end{enumerate} 
\end{enumerate}
Here, $E_{m,n,\mu}^-$ is a quotient of $E_{m+1,n+1}^-$ by a subgroup
isomorphic to $\bbG_{a(1)}$, and only depends on the
image of $\mu$ in $k^\times/(k^\times)^2$.\qed
\end{theorem}

\begin{definition}
\label{de:witt} 
The supergroup schemes of type III are said to be \emph{Witt elementary}. 
\end{definition} 

The role played by elementary supergroup schemes is explained by the following analogue from \cite{Benson/Iyengar/Krause/Pevtsova:bikp5} of the theorems of Quillen and Chouinard  for finite groups. 

\begin{theorem}
\label{th:super-elem-detect}
Let $G$ be a unipotent finite supergroup scheme over a field $k$ of characteristic $p\ge 3$. Then the following hold:
\begin{enumerate}[\quad\rm(i)]
\item 
An element $x\in H^{*,*}(G,k)$ is nilpotent if and only if for every extension field $K$ of $k$ and every elementary sub-supergroup scheme $E$  of $G_K$, the restriction of $x_K$, the image of $x$ in $H^{*,*}(G_K,K)$, to $H^{*,*}(E,K)$ is nilpotent.
\item 
A $kG$-module $M$ is projective if and only if for every extension field $K$ of $k$ and every elementary sub-supergroup scheme $E$ of $G_K$, the restriction of $M_K$ to $E$ is projective.\qed
\end{enumerate}
\end{theorem}

Fortunately, if we ignore the comultiplicative structure of these elementary supergroup schemes, the $\bbZ/2$-graded  algebra structure is  easy to
describe. The list corresponds to the one from Theorem~\ref{th:elemclass}.

\begin{proposition}
\label{pr:alg}
If $E$ is an elementary supergroup scheme over $k$ then the algebra $kE$ is isomorphic to one of the following:
\begin{enumerate}[\quad\rm(i)]
\item a tensor product of copies of $k[s]/(s^p)$,
\item a tensor product of copies of  $k[s]/(s^p)$ and  one copy of  $k[\sigma]/(\sigma^2)$,
\item a tensor product of copies of  $k[s]/(s^p)$  and one copy of $k[s,\sigma]/(s^{p^m},s^p-\sigma^2)$, where $m\ge 1$,
\end{enumerate}
with $|s|$ even and $|\sigma|$ odd. 
\end{proposition}

In particular, we have 
\begin{equation}
\label{eq:KEmn-} 
kE_{m,n}^- \cong \frac{k[s_1,\dots,s_n,\sigma]}{(s_1^p,\dots,s_{n-1}^p,s_n^{p^m},s^p-\sigma^2)}
\end{equation}
with coproduct defined by
\begin{align*} 
\Delta(s_i)&=S_{i-1}(s_1\otimes 1,\dots, s_i\otimes 1,
1\otimes s_1,\dots,1\otimes s_i) & (i\ge 1) \\
\Delta(\sigma)&=\sigma\otimes 1 + 1\otimes \sigma,
\end{align*}
where the $S_i$ are the maps coming from the comultiplication  in $\bbG_a$. 

Although an elementary supergroup scheme is not  necessarily commutative as a $\bbZ/2$-graded algebra, because of the factors of type \ref{pr:alg}\,(iii),
it is nonetheless commutative in the ungraded sense.

\begin{remark}
\label{rk:elementaries}
For non-unipotent finite supergroup schemes, these elementary supergroup schemes are definitely not sufficient for detection.  Drupieski and Kujawa~\cite{Drupieski/Kujawa:2019a}  introduced a slightly more general set of supergroup schemes which they show suffice for $\GL(a|b)_{(r)}$.
Following their notation, if $f$ is a \emph{$p$-polynomial}, meaning a polynomial  of the form $f(t)=\sum_i a_it^{p^i}$,  we shall write $\bbM_{n;f}$ 
for the supergroup scheme defined by replacing $s_n^{p^m}$ by $f(s_n)$ in~\eqref{eq:KEmn-}, and with the same comultiplication as $E_{m,n}^-$. Similarly, $\bbM_{n;f,\mu}$ is obtained in the same way from $E_{m,n-1,\mu}^-$.  The only unipotent ones among these are
our $E_{m,n}^-$ and $E_{m,n,\mu}^-$.  These are all quotients of a  supergroup scheme $\bbM_n$. Set
\begin{align*} 
k\bbM_n &\cong
\frac{k[s_1,\dots,s_{n-1},\sigma][[s_n]]}{(s_1^p,\dots,s_{n-1}^p,\sigma^2-s_n^p)}, \\
k\bbM_{n;f,\mu}&\cong \frac{k[s_1,\dots,s_n,\sigma]}{(s_1^p,\dots,s_{n-1}^p,f(s_n)+\mu s_1,\sigma^2-s_n^p)}\,.
\end{align*}
Thus $k\bbM_{n;f,\mu}$ is the group ring over $k$ of the corresponding supergroup scheme.  We shall suppress the field from the notation, so that we also regard $\bbM_n$ as a profinite supergroup scheme over any extension field $K$ of $k$. Similarly, if $f$ is a $p$-polynomial with coefficients in $K$, and $\mu\in K$, we shall write $\bbM_{n;f,\mu}$. Finally, if $f=t^{p^m}$ and $\mu=0$, we shall just write $\bbM_{n;m}$; this is the same as $E_{m,n}^-$.
\end{remark}

\section{\texorpdfstring
{$\pi$-points, $\pi$-supports and rank varieties}
{Ï-points,$\pi$-supports and rank varieties }}

To generalise the theory of $\pi$-points from finite group schemes to  finite supergroup schemes, instead of flat maps from $k[t]/(t^p)$, we
consider  maps of finite flat (or equivalently, projective) dimension from the superalgebra 
\[ 
A_k\colonequals \frac{k[t,\tau]}{(t^p-\tau^2)}\qquad (\tau \text{ odd, } t \text{ even}) 
\] 
We aim for theorems about unipotent finite supergroup schemes, but we stay more general for now in the interest of later developments, and also because we shall need to discuss $\GL(a|b)_{(n)}$ as part of the proof. So we shall include the elementary supergroup schemes  $\bbM_{n;f,\mu}$ from Remark~\ref{rk:elementaries}, but note that if $G$ is unipotent then the only ones that occur are the ones listed at the beginning of Section~\ref{se:elem}.

We view $A_k$ as a cocommutative Hopf superalgebra over $k$ with $\tau$ and $t$ primitive:
\[
\Delta(\tau)=\tau\otimes 1 + 1\otimes \tau\quad\text{and}\quad \Delta(t)=t\otimes 1 + 1\otimes t\,.
\]
This defines a homomorphism of algebras since
\begin{align*} 
\Delta(\tau^2)&=(\tau\otimes 1 + 1\otimes \tau)^2\\
&=\tau^2\otimes 1+\tau\otimes \tau -\tau\otimes \tau + 1\otimes \tau^2 \\
&= t^p\otimes 1 + 1\otimes t^p\\
&=\Delta(t^p).
\end{align*}
The cohomology ring of the $k$-algebra $A_k$ is easy to compute:
\[ 
\Ext^{*,*}_{A_k}(k,k)= k[\eta]\otimes \Lambda(u) 
\]
where $|\eta|=(1,1)$ and $|u|=(1,0)$. In particular,  modulo its radical it is a domain.

\begin{definition}
A \emph{$\pi$-point} $\alpha$ of  a finite supergroup scheme $G$ is given as follows. We choose an extension field $K$ of $k$, an elementary sub-supergroup scheme $E$ of $G_K$, and a map of superalgebras, but not necessarily respecting the coproduct,
\[ 
\alpha\colon A_K \to KE \subseteq KG_K 
\]
of finite flat dimension. 
\end{definition}

We put an equivalence relation on $\pi$-points, analogous to the one in Section~\ref{se:pi}.

\begin{definition}
We say that $\pi$-points $\alpha\colon A_K\to KG_K$ and $\beta\colon A_L\to LG_L$ are equivalent if, for all finite dimensional $kG$-modules $M$, the module $\alpha^*(M_K)$ has finite flat dimension if and only if $\beta^*(M_L)$ has finite flat dimension. We write $\Pi(G)$ for the set of equivalence classes of $\pi$-points of $G$.
\end{definition}

\begin{definition} 
We say a $\pi$-point $\alpha$ is \emph{$K$-rational} if it is defined over the field $K$, so that it is a map  $\alpha\colon A_K \to KG_K$.
\end{definition}

If $E$ is an elementary sub-supergroup scheme of $G_K$, and $\alpha\colon A_K\to KE$ is a $\pi$-point, then the radical of the kernel of restriction
\[ 
H^{*,*}(G,k) \subseteq H^{*,*}(G_K,K) \to H^{*,*}(E,K)  \to \Ext^{*,*}_{A_K}(K,K)
 \]
is a prime ideal $\fp(\alpha)$, for the target, modulo its radical, is a domain.

\begin{lemma}
If $\alpha$ and $\beta$ are equivalent $\pi$-points, then $\fp(\alpha)=\fp(\beta)$.
\end{lemma}

\begin{proof}
A proof of this result is given in  Proposition~6.8 of~\cite{Benson/Iyengar/Krause/Pevtsova:bikp7} for elementary supergroup schemes; it
involves Carlson's $L_\zeta$ modules. The same argument applies without any change to the general case. 
\end{proof}

\begin{definition}
 
By the preceding lemma, one gets a well-defined map
\[ 
\Phi_G \colon \Pi(G) \lra \Proj H^{*,*}(G,k) 
\]
where $\alpha$ is sent to $\fp(\alpha)$. By \cite[Theorem~6.9]{Benson/Iyengar/Krause/Pevtsova:bikp7}, see also \cite[\S 8]{Benson/Iyengar/Krause/Pevtsova:bikp7},  it is bijective when $G$ is an elementary supergroup scheme. See Section~\ref{se:homeo} for the general case.
\end{definition}

\begin{definition}
\label{de:phi}
Let $G$ be a finite \emph{unipotent} supergroup scheme over a field $k$ of characteristic $p\ge 3$. The \emph{$\pi$-support} of a $kG$-module $M$, denoted $\pisupp_G(M)$, is the subset of $\Pi(G)$ consisting of equivalences classes of $\pi$-points  $\alpha\colon A_K \to KG$ such that $\alpha^*(M_K)$ has infinite flat dimension.
This is well-defined by design.
\end{definition}

If $E$ is an elementary supergroup scheme, then by Proposition~\ref{pr:alg}  the group algebra $kE$ is isomorphic to  one of the following algebras:
\begin{enumerate} 
\item[I.]
$kE\cong k[s_1,\dots,s_n]/(s_1^p,\dots,s_n^p)$,  
\item[II.] $kE\cong k[s_1,\dots,s_n,\sigma]/(s_1^p,\dots,s_n^p, \sigma^2)$,
\item[III.] $kE \cong k[s_1,\dots,s_n,\sigma]/(s_1^p,\dots,s_{n-1}^p,s_n^{p^m},s_n^p-\sigma^2)$  
for some $m \ge 2, n\ge 1$.
\end{enumerate} 
In each case, a set of representatives for equivalence classes of $\pi$-points of $E$ can be prescribed by choosing an  extension field $K$ of $k$ and a point  $0\ne \lambda = (\lambda_1,\dots,\lambda_{n+1})$ in $\bbA^{n+1}(K)$. The $\pi$-point
\[
\alpha_{\lambda}\colon A_K\to KE_K 
\]
 is as follows: In case I, which is covered by Carlson's theory of rank varieties, set
\[
\alpha_\lambda(t) = \lambda_1 s_1 + \dots + \lambda_n s_n \quad\text{and}\quad  \alpha_\lambda(\tau)= 0\,.
\]
In the second case, set
\[
\alpha_\lambda(t) = \lambda_1 s_1 + \dots + \lambda_n s_n \quad\text{and}\quad  \alpha_\lambda(\tau)= \lambda_{n+1}\sigma\,.
\]
 See \cite[\S 8]{Benson/Iyengar/Krause/Pevtsova:bikp7}.  In the last case, the $\pi$-point is defined by
\begin{equation}
\label{eq:lambda}
\begin{aligned} 
\alpha_\lambda(t) &= \lambda_1 s_1 + \dots +
\lambda_{n-1}s_{n-1}+\lambda_{n-1}s_{n-1}+
\lambda_n s_n^{p^{m-1}}+\lambda_{n+1}^2 s_n, \\
\alpha_\lambda(\tau)&= \lambda_{n+1}^p\sigma.
\end{aligned}
\end{equation}
Note that 
\[ 
(\alpha_\lambda(t))^p=\lambda_{n+1}^{2p}s_n^p =\lambda_{n+1}^{2p}\sigma^2=(\alpha_\lambda(\tau))^2 
\] 
in $KE_K$, so this does indeed define a homomorphism of algebras.

The result below, which reproduces \cite[Theorem~4.12]{Benson/Iyengar/Krause/Pevtsova:bikp7},  extends Theorem~\ref{th:BCR}.

\begin{theorem}
\label{th:super-Dade}
Let $E$ be an elementary supergroup scheme defined over a field $k$. An  $E$-module $M$ is projective if and only if $\alpha_\lambda^*(M_K)$ has finite flat dimension for every extension field $K$ of $k$ and every $0\ne \lambda\in \bbA^{n+1}(K)$. \qed
\end{theorem}

Combining with Theorem~\ref{th:super-elem-detect}  yields that $\pi$-points detect projectivity of modules.

\begin{theorem}
\label{th:detection}
Let $G$ be a unipotent finite supergroup scheme over a field $k$ of characteristic $p\ge 3$. A $kG$-module $M$ is projective if and
only if $\pisupp_G(M)=\varnothing$. 
\end{theorem}

\begin{proof}
If $M$ is projective, $K$ is an extension field of $k$, and $E$ is a sub-supergroup scheme of $G_K$, then $\res_{G_K,E}(M_K)$ is projective. So if $\alpha\colon A_K\to E \to G_K$ is a $\pi$-point then $\alpha^*(M_K)$ has finite flat dimension. Thus $\pisupp_G(M)$ is empty.

Conversely, if $\pisupp_G(M)=\varnothing$ and $E$ is a sub-supergroup scheme of $G_K$ for some extension field $K$, then  $\res_{G_K,E}(M_K)$ has empty $\pi$-support. By Theorem~\ref{th:super-Dade}, $\res_{G_K,E}(M_K)$ is projective for every such $K$ and $E$. It follows from
Theorem~\ref{th:super-elem-detect} that $M$ is a projective $kG$-module.
\end{proof}

A crucial property of $\pi$-support is the tensor product formula.

\begin{theorem}
\label{th:tensor}
Let $G$ be a unipotent finite supergroup scheme over a field $k$ of characteristic $p\ge 3$, and let $M$ and $N$ be $kG$-modules.  As subsets of $\Pi(G)$ we have
\[ 
\pisupp_G(M\otimes_k N) = \pisupp_G(M) \cap \pisupp_G(N). 
\]
\end{theorem}

\begin{proof}
For any elementary sub-supergroup scheme $E$ of $G_K$, the restriction functor  
\[
\res_{G_K, E}\colon \Mod G_K \to \Mod E
\]
commutes with tensor product.  Hence, it suffices to prove the formula for elementary supergroup schemes. 
This is done in \cite[Theorem~7.6]{Benson/Iyengar/Krause/Pevtsova:bikp7}. 
\end{proof}

For finite-dimensional modules over elementary supergroups $\pi$-support  has a ``rank variety" interpretation, analogous to  Carlson's original construction
recalled in Definition~\ref{def:mcV}. 

Let $k[Y_1,\dots,Y_{n+1}]$ be the coordinate ring of $\bbA^{n+1}(k)$.  Since the map 
\[
\Phi_E\colon P(E) \to \Proj H^{*,*}(E,k) \cong \mathbb P^n (k)
\]
is a homeomorphism, for any $\fp\in\Proj k[Y_1,\dots,Y_{n+1}]$ there is a ``generic" point  $\lambda\in\bbA^{n+1}(K)$ such that $\Phi_E(a_\lambda) = \fp$ where $a_\lambda$  is the $\pi$-point defined in \eqref{eq:lambda}. 

\begin{definition}
Let $k$ be a field of characteristic $p\ge 3$, and let $E$ be a Witt elementary supergroup scheme over $k$. If $M$ is an $E$-module, we define $\mcV_E^r(M)$ to be the set of homogeneous primes $\fp\in\Proj k[Y_1,\dots,Y_{n+1}]$ such that if $\lambda\in\bbA^{n+1}(K)$ is generic for $\fp$ then
$\alpha^*_\lambda(M_K)$ has infinite flat dimension as a $KA_K$-module.
\end{definition}

It follows from \cite{Benson/Iyengar/Krause/Pevtsova:bikp7} that $\Phi_E$ 
takes the $\pi$-support of $M$ bijectively to $\mcV_E^r(M)$. 

\bigskip

Here is a free resolution of the trivial $A_K$-module $K$:
\[ 
\cdots \to \Pi A_K \oplus A_K 
\xrightarrow{\left(\begin{smallmatrix} \tau & t \\ \!-t^{p-1} & -\tau
\end{smallmatrix}\right)} 
A_K \oplus \Pi A_K
\xrightarrow{\left(\begin{smallmatrix} \tau & t \\ \!-t^{p-1} & -\tau
\end{smallmatrix}\right)} 
\Pi A_K \oplus A_K 
\xrightarrow{(\tau,t)} A_K \to 0 
\]
This is periodic from degree one onwards, with period one. So
for a $\pi$-point $\alpha\colon A_K\to KE_K$ and a $E$-module $M$, 
it follows that $\alpha^*(M_K)$ has finite projective dimension as a
$A_K$-module if and only if $\Ext^1_{A_K}(K,\alpha^*(M_K))=0$.
Taking homomorphisms from the above resolution to $\alpha^*(M_K)$, 
this is true if and only if
\[ 
\Pi M_K \oplus M_K
\xrightarrow{\left(\begin{smallmatrix} \alpha_\lambda(\tau) & \alpha_\lambda(t) \\ 
-\alpha_\lambda^{p-1}(t) & 
  -\alpha_\lambda(\tau)\end{smallmatrix}\right)}
M_K\oplus \Pi M_K
\xrightarrow{\left(\begin{smallmatrix} \alpha_\lambda(\tau) & \alpha_\lambda(t) \\ 
-\alpha_\lambda^{p-1}(t) & 
  -\alpha_\lambda(\tau)\end{smallmatrix}\right)} 
\Pi M_K \oplus M_K 
\]
is exact. In particular, if $M$ is a finitely generated $E$-module, of
dimension $d$, we can think of the action of
\begin{equation}\label{eq:matrix}
\left(\begin{smallmatrix}  \alpha_\lambda(\tau) & \alpha_\lambda(t) \\ 
-\alpha_\lambda^{p-1}(t) & 
  -\alpha_\lambda(\tau)\end{smallmatrix}\right) 
\end{equation}
on a direct sum of two copies of $M_K$ as a 
$2d \times 2d$ matrix whose square is zero. 
Exactness is therefore just the condition that 
this matrix has rank equal to $d$, which is the largest possible
for a square zero matrix of this size. This condition 
fails if and only if every $d\times d$ minor of the matrix
vanishes. In particular, this is a set of homogeneous 
polynomial conditions, and therefore defines a 
closed homogeneous subvariety of $\bbA^{n+1}(K)$.
So the following is an analogue of Carlson's rank variety.

\begin{definition}
If $k$ is algebraically closed, and $M$ is a finitely generated
$E$-module of dimension $d$, we define the \emph{rank variety}
$V^r_E(M)$ to be $\{0\}$ together with the set of 
$\lambda\in\bbA^{n+1}(k)$ such that $\alpha_\lambda^*(M)$ 
has infinite flat dimension. Equivalently, $V^r_E(M)$ is the
set of points $\lambda$ where the $2d\times 2d$
matrix representing the action of~\eqref{eq:matrix} on $M$ 
has rank strictly less than $d$.
\end{definition}

If we extend the field from $k$ to $K$, the polynomial equations defining 
the variety $V^r_{E_K}(M_K)$ are exactly the same as those defining
$V^r_{E}(M)$. It follows that a prime $\fp$ is in $\mcV^r_E(M)$ if and
only if a generic point for $\fp$ in a suitable $\bbA^{n+1}(K)$ is
contained in the zero set of these polynomials.
This happens if and only if $\fp$ is contained 
in the radical ideal defining $V^r_E(M)$. 

\begin{theorem}
For $M$ a finite dimensional $kG$-module, 
$\mcV^r_E(M)$ is the Zariski closed subset of $\Proj k[Y_1,\dots,Y_{n+1}]$
defined by the rank variety $V^r_E(M)$. 
\end{theorem}

On the other hand, it can be shown that for infinitely generated 
$E$-modules, every subset of $\Proj k[Y_1,\dots,Y_{n+1}]$ occurs
as $\mcV^r_E(M)$ for suitably constructed $M$.


\section{\texorpdfstring
{$\Phi_G$ is a homeomorphism}
{$\Phi_G$ is a homeomorphism }}
\label{se:homeo}

In this section we show that the map $\Phi_G\colon \Pi(G) \to \Proj H^{*,*}(G,k)$, introduced in Definition~\ref{de:phi}, is bijective for a unipotent finite supergroup scheme $G$.  Surjectivity will  follow from Theorem~\ref{th:super-elem-detect} whereas for injectivity  we have to recall some recent results of Drupieski and Kujawa. We shall use supergroup schemes $\bbM_n$ and $\bbM_{n;f,\mu}$ introduced in  Remark~\ref{rk:elementaries}.

Let $\Hom_{\sgs/k}(\bbM_n,G)$ be a functor from commutative $k$-superalgebras to sets defined as 
\[ \Hom_{\sgs/k}(\bbM_n,G)(R) = \Hom_{\sgs/R}(\bbM_{n,R}, G_R).\] 

\begin{definition}
It is shown in ~\cite[Theorem~3.3.6]{Drupieski/Kujawa:2019a} that if $G$ is any affine supergroup scheme of finite type then 
$\Hom_{\sgs/k}(\bbM_n,G)$ has the structure of a connected affine superscheme of finite type over $k$, which we denote
$\bfV_n(G)$. Similarly, $\Hom_{\sgs/k}(\bbM_{n;f,\mu},G)$ is an affine
sub-superscheme of  $\bfV_n(G)$, which we denote $\bfV_{n;f,\mu}(G)$.
In the case $f=t^{p^m}$ and $\mu=0$, we shall write $\bfV_{n;m}(G)$.
\end{definition} 
The $K$-points of $\bfV_n(G)$ are maps $\bbM_{n,K} \to G_K$ as
above. In particular, we can identify $\bfV_n(G)$ with $\bfV_n(G_{(n)})$.

In \S6.2 of~\cite{Drupieski/Kujawa:2019a}, the authors construct  a map of affine superschemes
\[ 
\Psi_G\colon \bfV_n(G)\to\Spec H^{*,*}(G,k) 
\]
coming from a map of coordinate rings
\[ 
\psi_G \colon H^{*,*}(G,k) \to k[\bfV_n(G)]. 
\]
Let $\GL(a|b)$ be the general linear supergroup; see, for example, \cite{Drupieski:2016a}. For a connected finite 
supergroup scheme $\GL(a|b)_{(n)}$ which is the Frobenius kernel of $\GL(a|b)$, we have the
following. There is a map constructed in~\S6 of~\cite{Drupieski/Kujawa:2019a}
\[ 
\bar\phi\colon k[\bfV_n(\GL(a|b))]_\ev= k[\bfV_n(\GL(a|b)_{(n)})]_\ev \to H^{*,0}(\GL(a|b)_{(n)}) 
\]
with the property that the composite
\[  
k[\bfV_n(\GL(a|b))]_\ev \xrightarrow{\bar\phi}  H(\GL(a|b)_{(n)}) \xrightarrow{\psi_{\GL(a|b)_{(n)}}}
k[\bfV_{n;f,\mu}(\GL(a|b))]_\ev 
\]
is the Frobenius morphism $F^n$ followed by the quotient map
\[  k[\bfV_n(\GL(a|b))]_\ev \to k[\bfV_{n;f,\mu}(\GL(a|b))]_\ev. \]
This has the following consequence, as pointed out in Theorem~6.2.3
of~\cite{Drupieski/Kujawa:2019a}.

\begin{theorem}\label{th:VnGLabn}
The image of map 
\[ \psi_{\GL(a|b)_{(n)}} \colon H^{*,*}(\GL(a|b)_{(n)},k) \to
  k[\bfV_{n;f,\mu}(\GL(a|b)_{(n)})] \]
contains the $p^n$th power of every element of $k[\bfV_{n;f,\mu}(\GL(a|b)_{(n)})]$.
The map
\[ \Psi_{\GL(a|b)_{(n)}}\colon \bfV_n(\GL(a|b)_{(n)}) \to \Spec H^{*,*}(\GL(a|b)_{(n)},k) \]
is injective on $K$-points for all
extension fields $K$ of $k$.
\qed
\end{theorem}

The following is the analogu  of~\cite[Theorem~5.2]{Bendel/Friedlander/Suslin:1997b}; see also~\cite[Proposition~3.8]{Friedlander/Pevtsova:2005a}.

\begin{theorem}\label{th:F-iso}
Let $G$ be connected finite supergroup scheme of height at most $n$.
Then for $m$ large enough, the image of the map 
\[ 
\psi_G\colon H^{*,*}(G,k) \to k[\bfV_{n;m}(G)] 
\]
contains the $p^n$th power of every element of $k[\bfV_{n;m}(G)]$. Consequently, the map $\Psi_G\colon \bfV_{n;m}(G) \to \Spec H^{*,*}(G,k)$ is injective.
\end{theorem}

\begin{proof}
Choose an embedding $G \to \GL(a|b)_{(r)}$ for suitable $a$, $b$ and
$r$. Then we have the following diagram
\[ 
\xymatrix@C=1.7cm{
H^{*,*}(\GL(a|b)_{(n)},k) \ar[r]^{\psi_{\GL(a|b)_{(n)}}}\ar[d] & k[\bfV_{n;m}(\GL(a|b)_{(n)})]\ar[d] \\
H^{*,*}(G,k) \ar[r]^{\psi_G} & k[\bfV_{n;m}(G)]}
\] 
The upper horizontal map surjects onto $p^n$th powers 
by Theorem~\ref{th:VnGLabn}, and the 
right hand vertical map is surjective since $\bfV_{n;m}(G) \to \bfV_{n;m}(\GL(a|b)_{(n)})$ is 
a closed embedding by \cite[Theorem~3.3.6]{Drupieski/Kujawa:2019a}. It follows that the lower
horizontal map surjects onto $p^n$th powers. 
\end{proof}

For the algebra $A = k[t, \tau]/(t^p-\tau^2)$, we have a map 
\begin{equation} 
\label{eq:distinguished} 
e\colon A \to k\bbM_n\to k\bbM_{n;m} 
\end{equation} 
given by $e(\tau)=\sigma$,
$e(t)=s_n$.  Let $\bbM_n \to G_K$ be a homomorphism of supergroup  schemes. 
Since $G$ is unipotent, the group algebra $KG_K$ is finite dimensional and local. 
Hence, any element in the augmentation ideal is nilpotent. Therefore, the 
homomorphism factors through
some $\bbM_{n;m}$. 
We compose to get a $\pi$-point of $G$:
\[ 
e_G\colon A_K \to K\bbM_n \to K\bbM_{n;m} \to KG_K. 
\] 
Thus we have maps
\[ 
\bbP \Hom_{\sgs/k}(\bbM_{n;m},G)\hookrightarrow
\bbP \Hom_{\sgs/k}(\bbM_n,G) \to \Pi(G) \xrightarrow{\Phi_G} \Proj H^{*,*}(G,k), 
\] 
where we denote by $\bbP X$ the projectivization of a conical affine scheme $X$.  It follows from \cite[\S6.2]{Drupieski/Kujawa:2019a} that the composition 
$\bbP \Hom_{\sgs}(\bbM_{n;m},G) \to  \Proj H^{*,*}(G,k)$ is induced by the  map $\Psi_{G}$.

\begin{proposition}
\label{pr:endo} 
The distinguished $\pi$-point $e\colon A \to k\bbM_{n:m}$ induces a bijection 
\[
e^*\colon \Proj k[\bfV_{n;m}(\bbM_{n;m})]  \to \Pi(\bbM_{n;m}). 
\] 
\end{proposition}

\begin{proof}
This follows from Lemma~3.3.2 in \cite{Drupieski/Kujawa:2019a}.
\end{proof}
\begin{remark}
Proposition~\ref{pr:endo} is effectively saying that any equivalence class of 
$\pi$-points of $\bbM_{n;m}$ has a unique (up to a scalar) representative of the form 
$\phi \circ e$ where $\phi$ is an endomorphism of the supergroup scheme  $\bbM_{n;m}$. 

This observation almost immediately extends to any finite connected unipotent supergroup scheme $G$. 
\end{remark} 

\begin{corollary} 
\label{co:surjective} 
For any finite connected unipotent supergroup scheme $G$ of height $n$, the $\pi$-point $e_G\colon A \to kG$ induces a surjection
\[
e_G^*\colon \Proj k[\bfV_{n}(G)]  \to \Pi(G). 
\]
\end{corollary}

\begin{proof} 
Let $\alpha\colon A_K \to KG_K$ be a $\pi$-point. Then $\alpha$ factors through  an elementary supergroup $i\colon E \hookrightarrow G_K$. We consider the case $E = E_{m,n}^-$, $m \geq 1$, the other three cases are similar.  

We factor $\alpha = \iota \circ \alpha^\prime\colon A_K \to KE_{m,n}^- \to  KG_K$.  By Proposition~\ref{pr:endo} there exists an endomorphism  $\phi\colon E_{m,n}^- \to E_{m,n}^-$ such that $\alpha^\prime$ is equivalent to  $\phi \circ e\colon A_K \to E_{m,n}^-$ as a $\pi$-point of $E_{m,n}^-$.  Hence, $\alpha$  is equivalent to $i \circ \phi \circ e = e_G^*(i \circ \phi)$ as a $\pi$-point of $G$. 
Hence   $e_G^*$ is surjective. 
\end{proof} 

\begin{corollary}
\label{co:connected} 
For $G$ a finite connected  unipotent supergroup scheme, the map $\Phi_G\colon \Pi(G) \to \Proj H^{*,*}(G,k)$ is injective. 
\end{corollary} 

\begin{proof} 
Let $G$ be of height $n$ and choose large enough $m$ such  that any map $\bbM_n \to G$ factors through $\bbM_{n:m}$. Then the 
composition 
\[
\xymatrix@C=1.7cm{\Psi_G\colon \Proj k[\bfV_{n:m}(G)] \ar[r]^-{e^*_G} &  \Pi(G) \ar[r]^-{\Phi_G} & \Proj H^{*,*}(G,k) }
\] 
is bijective by Theorem~\ref{th:F-iso} and the first map is surjective by 
Corollary~\ref{co:surjective}. Hence, $\Phi_G$ is injective.  
\end{proof} 

To prove the main result, Theorem~\ref{th:Phi-bijective}, of this section we wish to extend  Corollary~\ref{co:connected} to all---not necessarily connected---unipotent finite supergroup schemes. The strategy for the  reduction from the general to the connected case   closely follows \cite{Friedlander/Pevtsova:2005a}. We outline  the key steps of the argument citing the proofs from  \cite{Friedlander/Pevtsova:2005a} where they apply verbatim. 

Let $\mathsf E$ be an elementary abelian $p$-group, which we can view as a 
supergroup scheme concentrated in even degree. Denote by
${\bf \sigma_{\mathsf E}} \in H^*(\mathsf E,k)$ the cohomology class defined as
\[
\sigma_\mathsf E = \prod\limits_{0\not = \xi \in H^1(\mathsf E,\mathbb F_p)} \beta(\xi),
\]
where $\beta$ denotes the Bockstein homomorphism.  The key property of $\sigma_\mathsf E$ 
is that as a function on $\Proj H^*(\mathsf E,k)$ it vanishes on any subvariety 
$\Proj H^*(\mathsf E^\prime, k) \subset \Proj H^*(\mathsf E,k)$ for a proper subgroup $\mathsf E^\prime$ of  $\mathsf E$. 

For $G$ a finite supergroup scheme with a group of connected components $\pi_0(G) = \mathsf E$ 
an elementary abelian $p$-group, we have a natural projection $\phi\colon G \to \mathsf E$.  We set 
\[ 
\sigma_G = \phi^*(\sigma_\mathsf E) \in H^{*,0}(G, k). 
\] 

The proof of the following result relies on the construction of the \emph{Evens norm}~\cite{Evens:1961a} and goes exactly as in \cite[4.4]{Friedlander/Pevtsova:2005a}.

\begin{proposition}
\label{pr:evens} 
Let $G = G^0 \rtimes \pi$ be a finite supergroup scheme with group of connected components $\pi$.  Let $\mathsf E$ be an elementary abelian $p$-subgroup of $\pi$.  Then we have an injective map 
\[
\Spec (H^{*,*}(G^0\rtimes \mathsf E,k)[\sigma_{G^0\rtimes \mathsf E}^{-1}])/N_\pi(\mathsf E) \to \Spec H^{*,*}(G,k) 
\]
induced by the embedding of supergroups $G^0\rtimes \mathsf E \hookrightarrow G$. \qed
\end{proposition}

\begin{proposition}
\label{pr:invariants} \cite[4.5]{Friedlander/Pevtsova:2005a}
Let $G=G^0 \rtimes \pi$ and let $(G^0)^\pi$ be the sub-supergroup
scheme of invariants of the connected component $G^0$ under the
action of $\pi$.  Then the  image of the map in cohomology induced by the
natural embedding of group schemes $(G^0)^\pi \times \pi
\hookrightarrow G^0 \rtimes \pi = G$,
\begin{equation}
\label{surj} H^{*,*}(G,k) \to H^{*,*}((G^0)^\pi \times \pi,k),
\end{equation}
contains the $p^{|\pi|}$th power of every element. \qed
\end{proposition}

\begin{corollary} 
\label{co:invariants}  
Let $G = G^0 \rtimes \mathsf E$ be a finite supergroup scheme with group of connected components $\mathsf E$. 
The natural map 
\[ 
\Proj H^{*,*}((G^0)^\pi \times \mathsf E, k) \to \Proj H^{*,*}(G,k) 
\] 
is an embedding of topological spaces. \qed
\end{corollary}

The final preparatory result that we need underlines the connection between rank and support varieties.  
Let $\alpha\colon A \to kG$ be a $k$-rational $\pi$-point and $\fp(\alpha)$ be the corresponding homogeneous 
prime ideal in $H(G,k)$. Since $\alpha$ is defined over $k$, $\fp(\alpha)$ defines a closed point on $\Proj H(G,k)$
and we can localise $H^{*,*}(G,M)$ at $\fp(\alpha)$. 
 
\begin{lemma}  
Let $G$ be a unipotent finite supergroup scheme defined over an algebraically closed field $k$, and $M$ a finite-dimensional $kG$-module. If $\alpha$ is a $k$-rational $\pi$-point, then $\alpha^*(M)$ has finite flat dimension if and only if $H^{*,*}(G,M)_{\fp(\alpha)} = 0$. 
\label{le:connection} 
\end{lemma} 

\begin{proof}  
This is similar to the second half of the proof of  \cite[Theorem~4.8]{Friedlander/Pevtsova:2005a}. 
\end{proof} 

\begin{theorem}
\label{th:Phi-bijective}
If $G$ is a unipotent finite supergroup scheme then the map 
\[ 
\Phi_G\colon \Pi(G) \to \Proj H^{*,*}(G,k) 
\] 
is bijective.  If $k$ is algebraically closed then for any finite-dimensional $kG$-module $M$,  the map
$\Phi_G$ takes equivalence classes of $k$-rational $\pi$-points in $\pisupp_G(M)$ 
bijectively onto the closed points of $\supp_G(M)$. 
\end{theorem}

\begin{proof}  
Surjectivity follows from Theorem~\ref{th:super-elem-detect}, together with~\cite[Theorem~6.9]{Benson/Iyengar/Krause/Pevtsova:bikp7} that shows that $\Phi_G$ is bijective for $G$ elementary. 

We have thus to show that if $\Phi_G(\alpha) = \Phi_G(\beta)$  for $\pi$-points 
$\alpha, \beta$, then  $\alpha$ is equivalent to $\beta$.  We first consider 
the case when $\alpha\colon A \to kG$ and $\beta\colon A \to kG$ are both defined 
over the ground field $k$ which we assume to be algebraically closed.  

In this case the proof is essentially the same as the proof of Theorem~4.6
of~\cite{Friedlander/Pevtsova:2005a} and proceeds in a series of reductions. 
We sketch the main steps here and refer the reader to~\cite{Friedlander/Pevtsova:2005a} for details.  

Since $k$ is perfect, we can write $G = G^0\rtimes \pi$ where $G^0$ is connected 
and $\pi$ is a finite group (the group of connected components of $G$). By definition, 
$\alpha$ factors through some elementary supergroup $E_\alpha = E_\alpha^0 \times 
\pi_0(E_\alpha)$ and similarly $\beta$ factors through $E_\beta = E_\beta^0 \times 
\pi_0(E_\beta)$. By choosing representatives $\alpha$, $\beta$ within their equivalence 
classes in such a way that $\pi_0(E_\alpha)$ and $\pi_0(E_\beta)$ are minimal, we can 
apply Proposition~\ref{pr:evens} to show that $\pi_0(E_\beta)$ is conjugate to 
$\pi_0(E_\alpha)$ by an element in $N_\pi(\pi_0(E_\alpha))$. Since conjugation 
preserves equivalence of $\pi$-points, we can now assume that both $\pi$-points 
factor through $G^0 \rtimes \pi_0(E_\alpha)$. Hence, we reduce to the case when 
$G=G^0 \rtimes \mathsf E$ with the group of connected components elementary abelian.

 Since $\alpha$, $\beta$ factor through elementary sub-supergroup schemes 
 of $G = G^0 \rtimes \mathsf E$, they both must factor through $(G^0)^\mathsf E 
 \times \mathsf E$.  Corollary~\ref{co:invariants} now implies that we can assume 
 that $G = (G^0)^\mathsf E \times \mathsf E$ which completes the second reduction 
 step. Finally, since the group algebra of $(G^0)^\mathsf E \times \mathsf E$ is 
 isomorphic to a group algebra of a connected unipotent finite supergroup 
 scheme, the statement follows from Corollary~\ref{co:connected}.  
 
 The statement about $\Phi_G$ identifying equivalences classes of $\pi$-points defined 
 over $k$ and closed points of $\supp_G(M)$ is easily seen to be equivalent to Lemma~\ref{le:connection}. 
 
 We finally prove injectivity of $\Phi_G$ for any $\pi$-points. 
 
 Let $\alpha$, $\beta$ be $\pi$-points such that $\fp(\alpha) = \fp(\beta)$.  Since extending scalars does not affect $\fp(\alpha)$, we can 
 assume that $\alpha, \beta\colon A_K \to KG_K$ are defined over the same field $K$ which is algebraically  closed.    Suppose $\alpha$ and $\beta$ are not equivalent. Then there exists  a finite dimensional $kG$-module $M$ such that $\alpha^*(M)$ has  infinite flat dimension but $\beta^*(M)$ does not, or vice versa.  Let $\fp(\alpha,K)$ be the radical of the kernel of the map 
 $\alpha^*\colon H(G_K,K) \to H(A_K,K)$  which is a point in $\Proj H(G_K,K)$  lying over $\fp(\alpha)$, and similarly for $\fp(\beta,K)$. Since $\Phi_{G_K}$ 
 takes $\pi$-points defined over $K$ to the $K$-closed points of $\supp_{G_K}(M_K)$, 
 we conclude that $\fp(\alpha, K) \in \supp_{G_K}(M_K)$. Hence, 
 $\Ann_{H(G_K,K)} H(G_K,M_K) \subset \fp(\alpha, K)$ (see Remark~\ref{rk:classical}); and similarly, 
 $\Ann_{H(G_K,K)} H(G_K,M_K) \not \subset \fp(\beta, K)$. But for an ideal $\mathfrak I$ of $H(G,k)$ 
 we have  $ \mathfrak I \subset \fp(\alpha) = \fp(\alpha, K) \cap H(G,k)$ if and only if $\mathfrak I \subset \fp_{\alpha,K}$. We conclude 
 that $\fp(\alpha) \not = \fp(\beta)$ which is a contradiction. Hence, $\Phi_G$ is injective. 
 \end{proof}

\begin{remark} 
We can endow $\Pi(G)$ with the structure of the topological space by choosing 
$\pisupp_G(M)$ to be a closed set exactly when $M$ is a finite-dimensional 
$kG$-modules. With this topology $\Phi_G$ becomes a homeomorphism. 
\end{remark} 

\section{Stratification}
\label{se:localising}

Knowing that $\Phi_G$ is bijective we can apply the techniques developed for finite group schemes to identify  $\pisupp$ and $\supp$ for \emph{all} $kG$-modules $M$, thereby proving the analogue of the result of Avrunin and Scott~\ref{th:AS} for unipotent supergroup schemes. The arguments for the results stated in this section are  similar to the ones for finite group schemes which appear in \cite{Benson/Iyengar/Krause/Pevtsova:2018a,Friedlander/Pevtsova:2007a}, so we give only a brief outline of how they go, referring the reader to those papers for details.

As always $k$ will be a field  of characteristic $p\ge 3$ and $G$ a unipotent finite supergroup scheme over $k$. We first compute $\pi$-supports of the local cohomology modules introduced in \S~\ref{se:local}. From now on we identity $\Proj H^{*,*}(G,k)$ with $\Pi(G)$, using Theorem~\ref{th:Phi-bijective}.

\begin{proposition}
\label{pr:suppGamma}
Let $G$ be a unipotent finite supergroup scheme. 
If $V$ is a specialisation closed subset of
$\Proj H^{*,*}(G,k)$ then 
\begin{enumerate}
\item $\pisupp_G(\gam_V(k))=V$, 
\item $\pisupp_G(L_V(k))$ is the complement of $V$, and
\item $\pisupp_G(\gam_\fp(k))=\{\fp\}$.
\end{enumerate}
\end{proposition}

\begin{proof}
The proof of this is the same as in Proposition~6.6  of~\cite{Friedlander/Pevtsova:2007a}. For the last statement, we choose specialisation closed subsets 
$V$ and $W$ as in Definition~\ref{def:Gamma-p}, and use the tensor product formula given in Theorem~\ref{th:tensor}.
\end{proof}

Given Theorems~\ref{th:super-elem-detect} and ~\ref{th:tensor}, and Proposition~\ref{pr:suppGamma}, one can argue as in the proof of 
\cite[Theorem~6.1]{Benson/Iyengar/Krause/Pevtsova:2018a} to get establish the result below.

\begin{theorem}
\label{th:supp=pisupp}
Let $G$ be a unipotent finite supergroup scheme. For any $kG$-module $M$ one has $\supp_G(M)=\pisupp_G(M)$. \qed
\end{theorem}

At this point cosupport can no longer be ignored: We write $\cosupp_G(M)$ for the cosupport of a $kG$-module $M$, defined as in \ref{def:biksupp}, using
the action of $H^{*,*}(G,k)$ on $\StMod(kG)$. The \emph{$\pi$-cosupport} of $M$ consists of equivalence classes of $\pi$-points $\alpha\colon A_K\to KG_K$ such that $\alpha^*(\Hom_k(K,M))$ has infinite flat dimension; compare with Definition~\ref{de:phi}. It has to be checked that this is well defined;  see
\cite[Theorem~4.12]{Benson/Iyengar/Krause/Pevtsova:bikp7} for the case of elementary supergroup schemes. Then, using \cite[Theorem~11.3]{Benson/Iyengar/Krause/Pevtsova:bikp5} one can verify that $\pi$-cosupport, like $\pi$-support, detects projectivity.

\begin{theorem}
\label{th:detection-cosupport}
Let $G$ be a unipotent finite supergroup scheme. A $kG$-module $M$ is projective if and
only if $\picosupp_G(M)=\varnothing$. \qed
\end{theorem}

Here is the analogue of the tensor-product formula~\ref{th:tensor} for $\pi$-support, and proved in the same way: reduce to the super elementary case and apply \cite[Theorem~7.6]{Benson/Iyengar/Krause/Pevtsova:bikp7}.

\begin{theorem}
\label{th:hom}
\pushQED{\qed }Let $G$ be a unipotent finite supergroup scheme over a field $k$ of characteristic $p\ge 3$, and let $M$ and $N$ be $kG$-modules.  As subsets of $\Pi(G)$ we have
\[ 
\cosupp_G\Hom_k(M,N) = \pisupp_G(M) \cap \picosupp_G(N).  \qedhere
\]
\end{theorem}

It is now a simple matter to establish that $\StMod(kG)$ is stratified
by the action of $H^{*,*}(G,k)$. As explained in
Section~\ref{se:local} this yields the classification of localising
subcategories of $\StMod(kG)$, namely, Theorem~\ref{th:main}.

\begin{theorem}
\label{th:minimality}
Let $G$ be a unipotent finite supergroup scheme over $k$.  For point $\fp$ in $\Proj H^{*,*}(G,k)$, the  localising subcategory $\gam_\fp\StMod(kG)$ of $\StMod(kG)$ is minimal. In particular, there is a bijection, defined by $\pisupp_G(-)$, between localising subcategories of $\StMod(kG)$ and subsets of $\Proj H^{*,*}(G,k)$.
\end{theorem}

\begin{proof}
Since $G$ is unipotent every localising subcategory of $\StMod(kG)$ is tensor ideal. Thus minimality of $\gam_\fp\StMod(kG)$ is tantamount to the statement that for any non-projective modules $M,N$ in this subcategory, the $kG$-module $\Hom_k(M,N)$ is not projective; see \cite[Lemma~3.9]{Benson/Iyengar/Krause:2011a}. 

We can verify this in two ways: Using Proposition~\ref{pr:suppGamma}, and Theorems~\ref{th:detection-cosupport} and \ref{th:hom}, one can mimic the proof of \cite[Theorem~6.1]{Benson/Iyengar/Krause/Pevtsova:2017a} to prove that 
\[
\picosupp_G(M)=\cosupp_G(M)\quad\text{for any $kG$-module $M$.}
\]
Fix $kG$-modules $M$ and $N$ with $\supp_G(M)=\{\fp\}=\supp_G(N)$. Then $\fp$ is also in $\pisupp_G(M)$, by Theorem~\ref{th:supp=pisupp}. Moreover, $\fp$ is in the cosupport of $N$ because $\supp_G(N)$ and $\cosupp_G(N)$ have the same maximal elements~\cite[Theorem~4.13]{Benson/Iyengar/Krause:2012b}, and hence also in $\picosupp_G(N)$. Then using Theorem~\ref{th:hom} one gets 
\[
\picosupp_G \Hom_k(M,N) = \pisupp_G(M)\cap \picosupp_G(N) = \{\fp\}
\]
Thus $\Hom_k(M,N)$ is not projective, as desired. 

Here is another way, which circumvents the cosupport detection theorem: Let $\fm$ be a closed point of $\Proj H^{*,*}(G,k)$. We argue as in~\cite[Proposition~6.3]{Benson/Iyengar/Krause/Pevtsova:2018a}: Choose a finite field extension $K$ of $k$ and a $\pi$-point $\alpha\colon A_K\to KG$ representing $\fm$. For any $kG$-module $M$ there is an isomorphism of $G_K$-modules
\[
\Hom_k(K,k)\otimes_k M \cong \Hom_k(K,M)
\]
Since the module on the left is a direct sum of copies of $M_K$, we deduce that $\fm$ is in $\supp_G(M)$ if and only if it is in $\cosupp_G(M)$; this is the crucial observation. For then given non-projective $M,N$ in $\gam_\fm\StMod(kG)$ it allows us to conclude that  $\fm$ is also in $\cosupp_G(N)$, and hence Theorem~\ref{th:hom} yields  
\[
\picosupp_G(\Hom_k(M,N)) =\{\fm\}\,.
\]
Applying the observation once again, we conclude that $\fm$ is in the $\pi$-support of $\Hom_k(M,N)$, so it is not projective, by Theorem~\ref{th:detection}.

This settles the desired minimality at closed points in $\Proj H^{*,*}(G,k)$. A technique of reduction to closed points from \cite[\S8]{Benson/Iyengar/Krause/Pevtsova:2018a}, see also~\cite[\S3]{Benson/Iyengar/Krause/Pevtsova:2019a}, allows us to  treat general prime ideals.  Note that the extra grading on $H^{*,*}(G,k)$ introduces no new complications because its projective spectrum is the same as that of the singly graded ring $H(G,k)$ discussed in Definition~\ref{def:ProjH**,Hbullet}.

This finishes the proof of the theorem.
\end{proof}

As noted above, one can prove Theorem~\ref{th:minimality} without first establishing that $\picosupp_G(M) = \cosupp_G(M)$ for a $kG$-module $M$. In fact one can use deduce this equality, and hence  that  $\pi$-cosupport detects projectivity, from the theorem above. This is what is done in \cite[Part~IV]{Benson/Iyengar/Krause/Pevtsova:2018a}, where the reader can find further applications of the theorem.


\section{\texorpdfstring{Local duality}{Local duality}}
\label{se:duality}

Let now $G$ be any (not necessarily unipotent) finite supergroup scheme over a field $k$ of positive characteristic $p\ge 3$. The group algebra $kG$ is Frobenius; thus there is an isomorphism of $kG$-modules
\[
\Hom_k(kG,k) \cong  kG\otimes_k \delta_G 
\]
where $\delta_G$ is a one-dimensional $kG$-module and a super analogue of the modular function for finite group schemes; see \cite[\S4]{Benson/Pevtsova:2020a}. In particular, it has a parity, denoted $\epsilon_G$, that records its internal degree. The isomorphism above leads to a duality on the stable category of $G$, namely for all $M,N$ in $\stmod(kG)$, there is a natural isomorphism
\begin{equation}
\Hom_k(\uHom_{kG}(M,N),k) \cong \uHom_{kG}(N, \Omega(M\otimes_k\delta_G))
\end{equation}
The isomorphism above can be deduced from  Auslander's defect formula; see \cite[\S4]{Benson/Iyengar/Krause/Pevtsova:2019a} where this is done for finite group schemes. When $G$ is a finite group, the isomorphism above is nothing but classical Tate duality. 

Building on the Tate duality theorem, one can mimic the arguments in \cite{Benson/Iyengar/Krause/Pevtsova:2019a}, or better yet \cite{Benson/Iyengar/Krause/Pevtsova:bikp6} that treats general finite dimensional Gorenstein algebras, to get:

\begin{theorem}
Set $H(G)\colonequals H^{*,*}(G,k)$ and fix  $\fp$ in $\Proj H(G)$. The $kG$-module $\gam_\fp(\delta_G)$ is a dualising object in $\gam_\fp\StMod(kG)$, in that, for any $kG$-module $M$ there is a natural isomorphism
\[
\tExt^{i,*}_{kG}(M,\gam_{\fp}(\delta_G)) \cong \Hom_{H(G)}(H^{*-d-i,*+\epsilon_G}(G,M),I(\fp))
\]
where $d=\dim H(G)/\fp$ and $I(\fp)$ is the injective hull of the residue field at $\fp$. \qed
\end{theorem}

From this one gets, for example, a local cohomology spectral sequence:
\[
H^{s,t,j}_\fp H^{*,*}(G,M)_\fp \Rightarrow H_{-s-t-d,j+\epsilon_G}(G,M\otimes\delta_G).
\]
 discovered by Greenlees in the context of finite groups~\cite{Greenlees:1995a, Greenlees/Lyubeznik:2000a}.

\newcommand{\noopsort}[1]{}
\providecommand{\bysame}{\leavevmode\hbox to3em{\hrulefill}\thinspace}
\providecommand{\MR}{\relax\ifhmode\unskip\space\fi MR }
\providecommand{\MRhref}[2]{%
  \href{http://www.ams.org/mathscinet-getitem?mr=#1}{#2}
}
\providecommand{\href}[2]{#2}

\end{document}